\documentclass{amsart}
\usepackage{graphicx, amsmath, amsfonts, amssymb, amsthm, theoremref, hyperref, array, booktabs, caption, xparse, enumitem, color, mathtools} 
\newtheorem{theorem}{Theorem}
\newtheorem{lemma}[theorem]{Lemma}
\newtheorem{corollary}[theorem]{Corollary}

\newtheorem{prop}{Proposition}
\newtheorem{remark}[theorem]{Remark}

\newcommand{\density}{\rho}
\newcommand{\critdens}{\density_\mathrm{c}}
\newcommand{\Z}{\mathbb{Z}}
\newcommand{\N}{\mathbb{N}}
\newcommand{\R}{\mathbb{R}}
\newcommand{\Carpet}{\mathfrak{C}}
\NewDocumentCommand{\pesc}{ o }{%
  p_{\mathrm{esc}}\IfValueT{#1}{(#1)}
}
\NewDocumentCommand{\preturn}{ o }{%
	p_{\mathrm{ret}}\IfValueT{#1}{({#1})}
}

\NewDocumentCommand{\Green}{ O{} }{%
  \ensuremath{\mathbf{G}_{#1}}%
}
\newcommand{\E}{\mathbf{E}}
\renewcommand{\P}{\mathbf{P}}
\newcommand{\Vertices}{\Z^d}
\newcommand{\Graph}{\Z^d}
\newcommand{\FiniteVertices}{V}
\newcommand{\SubFiniteVertices}{U}
\newcommand{\s}{\mathfrak{s}}
\newcommand{\lambdasleep}{\lambda_\s}
\newcommand{\lambdajump}{\lambda_\mathrm{J}}
\newcommand{\psleep}{p_\s}
\newcommand{\Jump}{\mathrm{J}}
\newcommand{\pjump}{p_\Jump}
\newcommand{\config}{\eta}
\newcommand{\configg}{\xi}
\newcommand{\initdistr}{\nu}

\newcommand{\Ch}{\mathsf{Ch}}

\NewDocumentCommand{\Visits}{o}{%
	\mathsf{Vis}%
	\IfValueTF{#1}{({#1})}{}%
}
\NewDocumentCommand{\Revisits}{o}{%
	\mathsf{Revis}%
	\IfValueTF{#1}{({#1})}{}%
}
\newcommand{\Weak}{\mathsf{W}}
\newcommand{\Strong}{\mathsf{S}}
\newcommand{\SymbStab}{\mathsf{S}}
\newcommand{\OdomSymb}{\mathsf{O}}
\NewDocumentCommand{\SpacedOp}{ O{} m m }{
	{#2#3}%
	\IfValueTF{#1}{^{#1}\!}{}%
}
\NewDocumentCommand{\Stab}{o}{\SpacedOp[#1]{}{ \SymbStab }}
\NewDocumentCommand{\WeakStab}{o}{\SpacedOp[#1]{\Weak}{ \SymbStab }}
\NewDocumentCommand{\StrongStab}{o}{\SpacedOp[#1]{\Strong}{ \SymbStab }}
\NewDocumentCommand{\Odom}{o}{\SpacedOp[#1]{}{ \OdomSymb }}
\NewDocumentCommand{\WeakOdom}{o}{\SpacedOp[#1]{\Weak}{ \OdomSymb }}
\NewDocumentCommand{\StrongOdom}{o}{\SpacedOp[#1]{\Strong}{ \OdomSymb }}

\newcommand{\odom}{m}
\newcommand{\weakodom}{\odom^\Weak}
\newcommand{\weakconfig}{\config^\Weak}
\newcommand{\strongodom}{\odom^\Strong}
\newcommand{\strongconfig}{\config^\Strong}

\newcommand{\instr}{\mathcal{I}}
\newcommand{\jumpinstr}[1]{\mathfrak{j}_{#1}}
\newcommand{\neutral}{\iota}
\newcommand{\origin}{0}
\newcommand{\sleepoutcome}{b}

\DeclareMathOperator{\Ber}{Bernoulli}
\DeclareMathOperator{\Geom}{Geom}

\NewDocumentCommand{\Fill}{o}{%
	\mathsf{Fill}%
	\IfValueTF{#1}{_{#1}}{}%
}
\newcommand{\ReturnEv}[1]{\mathsf{Return}_{#1}}
\newcommand{\Emp}[1]{\mathsf{Emp}_{#1}}

\newcolumntype{D}{>{\raggedright\arraybackslash}p{1.3cm}}   
\newcolumntype{C}{>{\raggedright\arraybackslash}p{3.2cm}}   
\newcolumntype{K}{>{\raggedright\arraybackslash}p{3.2cm}}   
\newcolumntype{O}{>{\raggedright\arraybackslash}p{3.55cm}}   

\title[Asymptotic critical density of activated random walk]{Asymptotic behavior of the critical density of activated random walk}
\author{Harley Kaufman}\address{Harley Kaufman, Department of Mathematics, Baruch College, City University of New York}
	\email{\texttt{Harkauf100@gmail.com} } 
\author{Josh Meisel}\address{Josh Meisel, Department of Mathematics, Graduate Center, City University of New York}
	\email{\texttt{jmeisel@gradcenter.cuny.edu}}

\begin{document}

\begin{abstract}
    We study the asymptotic behavior of the critical density of the activated random walk model as the sleep rate $\lambda$ tends to $0$ and $\infty$. For large $\lambda$, we prove new lower bounds in dimensions 1 and 2, showing that in one dimension the critical density approaches $1$ superpolynomially fast. For small $\lambda$, we prove a new lower bound in two dimensions for how fast the critical density vanishes. We also obtain the first-order approximation for transient walks in both regimes.
\end{abstract}
\maketitle

\section{Introduction}

Activated random walk (ARW) is an interacting particle system taking place on the integer lattice $\Z^d$. A descendant of the abelian sandpile and stochastic sandpile models, it is intended as a proof-of-concept for the theory of self-organized criticality from physics \cite{bak1987soc, dickmanSOC, dickman2010activated}. In ARW, particles can be in one of two states, active and sleeping. At the start, all particles are active, with the configuration of particles $\config \colon \Vertices \to \N$ distributed according to some translation-ergodic \emph{initial distribution} $\initdistr$. Active particles perform independent continuous-time random walks at \emph{jump rate} $1$, governed by the jump distribution $P \colon \Vertices \times \Vertices \to [0,1]$. In this article, we assume the jump distribution is translation-invariant, that is $P(x,y)$ is a function of $y - x$. Active particles fall asleep at \emph{sleep rate} $\lambda > 0$. Sleeping particles remain inert, but are activated when visited by another particle. In particular, a particle that falls asleep at a site occupied by any other particle is immediately reactivated. We say the system \emph{fixates} if each particle eventually falls asleep and is not woken up again. Otherwise we say the system \emph{stays active}. 

The model is known to undergo a phase transition with respect to the initial density $\density \coloneqq \E|\config(\origin)|$, treating $\lambda$ as a fixed constant: for any nearest-neighbor walk $(\Z^d, P)$ and sleep rate $\lambda > 0$, there exists a critical density $\critdens = \critdens(\Z^d, P, \lambda)$ such that, for any translation-ergodic initial distribution with density $\density$, the system fixates a.s.\ if $\density < \critdens$ and a.s.\ stays active if $\density > \critdens$ \cite{rs, universalityTheorem, AmirGurel-Gurevich10, rt}. As noted in \cite{rollaSurvey}, such a critical density exists for general \emph{unimodular walks}, defined in Section~\ref{subsec:extensions}, which includes all translation-invariant walks on $\Z^d$.

    The critical density is non-decreasing and continuous in $\lambda$ \cite{rs, continuity}. By the collective works of \cite{AmirGurel-Gurevich10, Shellef10, rs, StaufferTaggi18} it is known that $0 < \critdens(\Z^d, P, \lambda) \le 1$ and that $$\lim_{\lambda \to \infty} \critdens = 1.$$
    The strict inequality $\critdens < 1$ and the limit 
    $$\lim_{\lambda \to 0} \critdens = 0$$
    have been shown in essentially complete generality \cite{Taggi2016, rt, BasuRiddhipratim2018NfCS, StaufferTaggi18, Taggi2019, ForienNicolas2024Apfa, AsselahAmine2024Tcdf, hu2022active}, specifically for simple symmetric random walk (SSRW) in all dimensions as well as all transient walks.
    
    The limiting $\lambda = 0$ case is a system of non-interacting random walks, and the limiting $\lambda = \infty$ case is a sort of Diffusion Limited Aggregation \cite{lawler1992internal}. Measuring the asymptotic rate of $\critdens$ as the system approaches these two extremes gives us information about the dying-out effects of the sleep-activation dynamics, and remains an ongoing area of investigation.

\subsection{Results}

We state our results here, and in Section~\ref{subsec:prior-results} provide further detail on prior asymptotic bounds and heuristics. For a quick glance, Table~\ref{tab:results} provides a full picture of the current state of progress for simple symmetric ARW on $\Z^d$, where we use the abbreviation $\critdens(\Z^d, \lambda)$ for the critical density. The Green's function is denoted $\Green[d] = \Green[d](\origin, \origin) \ge 1$, the expected number of origin hits of a random walk started at the origin, and the probability of escape for $d \ge 3$ is denoted $\pesc(d) \coloneqq 1/\Green[d]$. We use the standard asymptotic notation $O$, $o$, $\Theta$, $\Omega$, and $\omega$ for upper, strict upper, sharp, lower, and strict lower bounds, respectively.

\begin{table}
\centering
\begin{minipage}{\textwidth}
\centering
\begin{tabular}{D C K O}
\toprule
{\normalsize\bfseries $\mathbf{d}$} & {\normalsize\bfseries Heuristic} & {\normalsize\bfseries Known bound} & {\normalsize\bfseries Our bound} \\
\midrule

\multicolumn{4}{c}{\normalsize Low sleep rate: $\critdens(\Z^d, \lambda)$ as $\lambda \to 0$} \\
\midrule
$1$ &
$\Theta(\sqrt{\lambda})$ &
$\Theta(\sqrt{\lambda})$  &
\\
$2$ &
$\Theta(\lambda\log(\lambda^{-1}))$ &
$O(\lambda \log^a(\lambda^{-1}))$\textsuperscript{†}  &
$\Omega(\lambda\log(\lambda^{-1}))$ \\
$\ge 3$ &
$\Theta(\lambda)$ &
$\Theta(\lambda)$  &
$\Green[d]\lambda + o(\lambda)$ \\
\midrule

\multicolumn{4}{c}{\normalsize High sleep rate: $1 - \critdens(\Z^d, \lambda)$ as $\lambda \to \infty$} \\
\midrule
$1$ &
$\Theta(\lambda^{-1})$\textsuperscript{‡} &
$e^{-O(\lambda)}$ &
$o(\lambda^{-n})$ for any $n \ge 1$ \\
$2$ &
$\Theta(\lambda^{-1})$ &
$\Omega((\lambda \log^2 \lambda)^{-1})$ &
$O((\lambda\log \lambda)^{-1})$ \\
$\ge 3$ &
$\Theta(\lambda^{-1})$ &
$\Theta(\lambda^{-1})$ &
$\pesc(d)\lambda^{-1} + o(\lambda^{-1})$ \\
\bottomrule
\end{tabular}
\caption{Heuristic, known, and new bounds for the asymptotic behavior of $\critdens(\Z^d, \lambda)$. The heuristics are taken from \cite{AsselahAmine2024Tcdf} and the known results are from \cite{AsselahRollaSchapira2022, AsselahAmine2024Tcdf, HoffmanChristopher2023APfA}, and \cite{Taggi2019}.}
\label{tab:results}
\end{minipage}

\vspace{1mm}

\begin{flushleft}
\footnotesize
\textsuperscript{†} $a$ is an unspecified positive constant. Our new bound implies that $a \ge 1$.\\
\textsuperscript{‡} As mentioned below, a recent vote showed some deviation from this heuristic in the community, including from one of its authors, who held the vote.
\end{flushleft}
\end{table}

In dimension 1, Hoffman, Richey and Rolla showed that there is at least an exponentially small gap between $\critdens(\Z, \lambda)$ and $1$, \cite{HoffmanChristopher2023APfA}, that is 
$$\critdens(\Z, \lambda) \le 1 - e^{-c\lambda}$$
for some constant $c > 0$ and all large $\lambda$. A vote was taken at the 2025 Activated Random Walk focused workshop at the Erdős Center in Budapest on whether $\critdens(\Z, \lambda) \to 1$ as $\lambda \to \infty$ polynomially or exponentially fast (or at some intermediate speed), resulting in a 3--3 tie. We partially resolve the question, showing superpolynomial convergence, contradicting the second author's vote for a polynomial rate.

\begin{theorem}\thlabel{thm:high-sleep-rate-d-1}
    As $\lambda \to \infty$,
    $$
        \critdens(\Z, \lambda) = 1 - o(\lambda^{-n})
    $$
    for all $n \ge 1$.
\end{theorem}

We obtain other high-sleep-rate bounds. Denote the Green's function for general walks by $\Green[(\Graph, P)] = \Green[(\Graph, P)](\origin, \origin)$ and let $\pesc = \pesc(\Z^d, P) \coloneqq 1/\Green[(\Z^d, P)]$ for transient $(\Graph, P)$. As a reminder, all jump distributions are assumed to be translation-invariant. 

\begin{theorem}\thlabel{thm:high-sleep-rate}
    As $\lambda \to \infty$,
	\begin{enumerate}[label=(\roman*), ref=(\roman*)]
		\item\label{thm:high-sleep-rate:Z2} $\critdens(\Z^2, \lambda) =  1 - O((\lambda \log \lambda)^{-1})$.
		\item\label{thm:high-sleep-rate:transient} $\critdens(\Graph, P, \lambda) = 1 - \pesc \lambda^{-1} + o(\lambda^{-1})$ if $(\Graph,P)$ is transient.
        \item\label{thm:high-sleep-rate:recurrent} $\critdens(\Graph, P, \lambda) = 1 - o(\lambda^{-1})$ if $(\Graph,P)$ is recurrent.
	\end{enumerate}
\end{theorem}

The lower bound acquired in dimension 2 matches the heuristic prediction of Asselah, Forien, and Gaudilli\`ere \cite{AsselahAmine2024Tcdf}, explained in the following section, and closely matches their rigorous $1 - \Omega((\lambda \log^2 \lambda)^{-1})$ upper bound. 

For low sleep rates, we prove the following.

\begin{theorem}\thlabel{thm:low-sleep-rate}
	As $\lambda \to 0$, 
	\begin{enumerate}[label=(\roman*), ref=(\roman*)]
		\item\label{thm:low-sleep-rate:Z2} $\critdens(\Z^2, \lambda) = \Omega(\lambda \log (\lambda^{-1}))$.
		\item\label{thm:low-sleep-rate:transient} $\critdens(\Graph, P, \lambda) = \Green[(\Graph, P)] \lambda + o(\lambda)$ for any transient walk $(\Graph,P)$. 
        \item \label{thm:low-sleep-rate:recurrent}
        $\critdens(\Graph, P, \lambda) =\omega(\lambda)$ for any recurrent walk $(\Graph,P)$.
	\end{enumerate}
\end{theorem}

Again, the lower bound in dimension 2 matches the heuristic from \cite{AsselahAmine2024Tcdf} and closely matches their upper bound of $O(\lambda \log^a(\lambda^{-1}))$ for some positive $a > 0$.

\subsection{Prior results}\label{subsec:prior-results}

\subsubsection{Low sleep rate}\label{subsec:low-lambda}
 
As $\lambda \to 0$, as in the setting of \thref{thm:low-sleep-rate}, the rate at which $\critdens$ vanishes is at least $\Omega(\lambda)$, due to the following universal lower bound on the critical density. Let 
 $$
    \lambdasleep \coloneqq \frac{\lambda}{1 + \lambda}, \qquad \lambdajump \coloneqq 1 - \lambdasleep = \frac{1}{1 + \lambda},
$$ 
be the \emph{normalized} sleep and jump rates, scaled to sum to $1$. Then
  \begin{equation}\label{eq:univ-lb}
 	\critdens \ge \lambdasleep
 \end{equation}
is known to hold quite generally, having been first demonstrated for SSRW on $\Z$ \cite{rs}, then for any walk on an amenable graph \cite{StaufferTaggi18}, and finally SSRW on any vertex-transitive graph \cite{Taggi2019}. For the latter two results, Stauffer and Taggi developed a technique we refer to as \emph{strong stabilization via successive weak stabilization}, which we heavily exploit throughout this article.
 
 Regarding upper bounds, an $O(\lambda^{1/4})$ rate was achieved for transient walks \cite{StaufferTaggi18} and $O(\sqrt{\lambda})$ for transient walks on amenable graphs \cite{Taggi2019}, with the latter bound later improved to $\Theta(\lambda)$ and to include SSRW on any vertex-transitive graph \cite{meanfield}. For dimension 1, the correct order $\Theta(\sqrt{\lambda})$ was achieved \cite{AsselahRollaSchapira2022}. In \cite{AsselahAmine2024Tcdf}, the heuristic argument is provided that at the critical density, a particle should encounter on average one other particle before falling asleep, leading to the following predictions for the order of $\critdens$ as $\lambda \to 0$ for simple symmetric ARW:

$$\left\{
\begin{array}{ll}
	\Theta( \sqrt \lambda) & \text{if }d=1 \\
	\Theta(\lambda\log (\lambda^{-1}))& \text{if }d=2\\
	\Theta(\lambda)& \text{if }d\ge 3.\\
\end{array}
\right.$$
They demonstrated the $\Theta(\lambda)$ prediction for $d \ge 3$ and in dimension 2 attained a close upper bound of $O(\lambda \log^a (\lambda^{-1}))$ for some positive value $a$. Inspired by the heuristic argument, we prove a rigorous lower bound for the critical density in \thref{lem:general-lb} below. 

Together with \thref{thm:low-sleep-rate} then, the picture for low sleep rates is nearly complete, with only a small gap in dimension 2 between orders $\Theta(\lambda \log (\lambda^{-1}))$ and $\Theta(\lambda \log^a (\lambda^{-1}))$.

\subsubsection{High sleep rate}\label{subsec:high-lambda}

For the behavior as $\lambda \to \infty$, from \eqref{eq:univ-lb} we get that 
\begin{equation*}\label{eq:small-gap}
    1 - \lambdajump \le \critdens \le 1.
\end{equation*}
We thus get 
\begin{equation}\label{eq:crude-high-lambda}
    \critdens = 1 - O(\lambda^{-1})  
\end{equation}
as a general lower bound on the rate of convergence. 

Regarding upper bounds, \eqref{eq:crude-high-lambda} was shown to be tight for positive-speed SSRW, for instance on a non-amenable graph \cite{StaufferTaggi18}. In dimension 1, as stated above \cite{HoffmanChristopher2023APfA} obtained the exponential bound 
$$
 \critdens(\Z, \lambda) = 1 - e^{-O(\lambda)}.
$$
In \cite{AsselahAmine2024Tcdf}, the authors attained the upper bounds 
$$
    \critdens(\Z^2, \lambda) = 1 - \Omega(\lambda^{-1} \log^{-2}\! \lambda)
$$
and for $d \ge 3$,
$$
    \critdens(\Z^d, \lambda) = 1 - \Omega(\lambda^{-1}\log^{-1}\! \lambda),
$$
with their heuristic generally predicting $\critdens = 1 - \Theta(\lambda^{-1})$. And although it is not explicitly spelled out, it can be ascertained from \cite{Taggi2019} that for any transient walk $P$ on an amenable graph $G$, 
$$
    \critdens(G, P, \lambda) = 1 - \Omega\big(\lambda^{-\lfloor \Green[(G, P)] \rfloor}\big),
$$
and thus $\critdens(G,P,\lambda) = 1 - \Theta(\lambda^{-1})$ whenever $\Green[(G, P)] < 2$, including then SSRW in any transient dimension. Note that we extend this result to any transient walk, including on graphs where the Green's function is high.

With the addition of Theorems~\ref{thm:high-sleep-rate-d-1} and~\ref{thm:high-sleep-rate}, the only remaining gaps are between superpolynomial and exponential in dimension 1, and between $1 - \Theta(\lambda^{-1} \log^{-1}\! \lambda)$ and $1 - \Theta(\lambda^{-1} \log^{-2}\! \lambda)$ in dimension 2.

\subsection{Extensions}\label{subsec:extensions}

For more general vertex-transitive graphs $G = (V,E)$, as stated in \cite{rollaSurvey}, an analogously defined critical density $\critdens(G,P,\lambda)$ exists if $(G,P)$ is unimodular, meaning $P$ is invariant under a transitive unimodular subgroup $\Gamma \le \operatorname{Aut}(G)$, ensuring that what is known as the mass transport principle holds for $\Gamma$. This includes for example SSRW on Cayley graphs such as regular trees. 

Our results consist mostly of lower bounds, which all use \thref{prop:fill}. As stated in \cite[Section 6]{meanfield}, the proposition extends to any amenable graph or positive-speed walk, including SSRW on a non-amenable graph.

The remaining results are the upper bounds in the equality statements for transient walks of \thref{thm:high-sleep-rate}~\ref{thm:high-sleep-rate:transient} and \thref{thm:low-sleep-rate}~\ref{thm:low-sleep-rate:transient}, which go through for general unimodular walks. For instance, the upper bound from \thref{thm:high-sleep-rate}~\ref{thm:high-sleep-rate:transient} states that $1 - \critdens(\Graph, P, \lambda) \le \pesc\lambda^{-1} + o(\lambda^{-1})$ as $\lambda \to \infty$ for any transient walk $(\Graph,P)$.

\subsection{Organization}

Section~\ref{sec:sketch} discusses the proof techniques and provides a sketch of the arguments. Section~\ref{section:site-wise} describes the site-wise representation of ARW as well as some important properties. Section~\ref{sec:weak-stab} defines \emph{weak} and \emph{strong stabilization}, and Section~\ref{subsec:strong-stab-via-weak-stab} describes the strong stabilization via successive weak stabilization procedure, proving a few useful relations.  In Section~\ref{sec:proofs-low-lambda} we prove \thref{thm:low-sleep-rate} and in Section~\ref{sec:proofs-high-lambda} we prove \thref{thm:high-sleep-rate} and then \thref{thm:high-sleep-rate-d-1}.

\section{Proof Overview}\label{sec:sketch}

Each of our arguments builds on, in different ways, the weak--strong stabilization framework from Stauffer and Taggi. As described in Section~\ref{subsec:strong-stab-via-weak-stab}, the framework's power lies in the perspective that each weak stabilization (except the last) produces an independent $\Ber(\lambdasleep)$ \emph{sleep trial}, with the true stabilization leaving a sleeping particle at the origin if and only if at least one of the $\Ch$ sleep trials succeeds. We add to prior work by performing a more detailed, quantitative analysis of the random variable $\Ch$.

In \cite{StaufferTaggi18} and \cite{Taggi2019}, the celebrated one-dimensional lower bound $\critdens(\Z, \lambda) \ge \lambdasleep$ from \cite{rs} was proven generally by noting that $\Ch \ge 1$ whenever the odometer at the origin is non-negative, so w.h.p.\ at supercritical density as the graph expands. The other results using the technique are the upper bounds on $\critdens$ from \cite{StaufferTaggi18, Taggi2019, meanfield}, all following from the fact that $\Ch$ is bounded in expectation by the Green's function, or in the latter case, from the slightly stronger inequality $\E[(\Ch - 1)^+] \le \Green[d] - 1$. 

We obtain more refined bounds through closer study of $\Ch$, especially in the high-sleep-rate regime, where we also examine the configuration after each weak stabilization during the procedure. To get the superpolynomial bound in \thref{thm:high-sleep-rate-d-1}, we iteratively bootstrap the improved bounds on the final density of the stable configuration, simultaneously using them to increase the density of particles after each weak stabilization. To our knowledge, this is the first time a bootstrapping technique has been applied to bound $\critdens$. 

We believe further quantitative study of the weak--strong stabilization technique could prove fruitful. 

\subsection{Proof sketch}

The proofs are given in reverse order, building in complexity. 

For \thref{thm:low-sleep-rate} on low sleep rates, the lower bounds are all consequences of the following lemma. Let $q=q(\Z^d, P, \lambda)$ denote the probability that a single particle starting at $\origin$ falls asleep at $\origin$, i.e.\ it is located there after $\Geom(\lambdasleep) - 1$ steps.

\begin{lemma}\thlabel{lem:general-lb}
    For any $\lambda > 0$, $$\critdens(\Z^d, P, \lambda) \ge q(\Z^d, P, \lambda).$$
\end{lemma}
\thref{lem:general-lb} follows from the observation that when a particle is jumped out of the origin after a sleep trial, if it returns without falling asleep, there will be another sleep trial. Thus, $\Ch$ dominates the number of excursions the particle completes without falling asleep. Furthermore, the bound is tight for transient walks, since the number of such excursions is close to the Green's function for low sleep rates.

For Theorems~\ref{thm:high-sleep-rate-d-1} and~\ref{thm:high-sleep-rate} on high sleep rates, the particle jumped out of $\origin$ can still perform a long random walk, now because it has a large carpet of sleeping particles around it that it can glide over due to the high critical density. For the one-dimensional result, we bound the probability the particle falls off the carpet by jointly considering the probability of an abnormally long excursion together with a sufficiently small carpet. We then iteratively bootstrap this improvement, also using it to increase the density of the carpet present after each successive weak stabilization, since the carpet may be part of the final stable configuration, depending on the outcomes of the sleep trials.


\section{The site-wise representation}\label{section:site-wise}

For a more detailed introduction, we direct the reader to \cite[Section 2]{rollaSurvey}. 

In the \emph{site-wise} or Diaconis-Fulton representation of ARW, there are random \emph{instructions} at each site, 
\[
    \instr = (\instr_x(k))_{x \in \Vertices, \, k \ge 0}.
\]
The instruction stacks $\instr_x = (\instr_x(k))_{k \ge 0}$ are independent and composed of sleep instructions $\s$ and jump instructions $\jumpinstr{y}$ for $y \in \Vertices$. The stack $\instr_x$ is i.i.d.\ with $\P(\instr_x(k) = \s) = \psleep$ and $\P(\instr_x(k) = \jumpinstr{y}) = \pjump \, P(x, y)$ for each $y \in \Vertices$. 

When stabilizing a finite set of vertices $\FiniteVertices \subset \Vertices$, killing particles that leave $\FiniteVertices$, we represent the state of the system by the particle \emph{configuration} and \emph{odometer}, respectively
\[
    \config \colon \FiniteVertices \to \N \cup \{\s\}, \qquad
    \odom \colon \FiniteVertices \to \N.
\]
The configuration value $\config(x)=\s$ indicates the presence of a single sleeping particle at site $x \in \FiniteVertices$ while $\config(x) \in \N$ indicates the number of active particles. The odometer value $\odom(x)$ counts how many instructions have been used at $x$. We call $x$ \emph{unstable} if it contains at least one active particle, and $\FiniteVertices$ unstable if some $x \in \FiniteVertices$ is. \emph{Toppling} an unstable site $x$ has the effect of updating the configuration according to instruction $\instr_x(\odom(x))$ and then increasing $\odom(x)$ by $1$. The \emph{abelian property} states that if one starts from state $(\config, \odom)$, and stabilizes by toppling unstable sites in any order until $\FiniteVertices$ is stable, the resulting state is always the same. It is composed of 
\[
    \Stab[\FiniteVertices](\config, \odom) \in \{0, \s\}^\FiniteVertices, \qquad \Odom[\FiniteVertices](\config, \odom) \in \N^\FiniteVertices,
\]
which we call the \emph{stable configuration} and \emph{stabilizing odometer} for $(\config, \odom)$, respectively.
\begin{lemma}[Abelian property]
    Fix finite $\FiniteVertices \subset \Vertices$, instructions $\instr$, configuration $\config$ on $\FiniteVertices$, and odometer $\odom$ on $\FiniteVertices$. Stabilizing from $(\config, \odom)$ in any order produces stable configuration $\Stab[\FiniteVertices](\config, \odom)$ and stabilizing odometer $\Odom[\FiniteVertices](\config, \odom)$.
\end{lemma}

Typically the odometer starts from zero, so we define $\Stab[\FiniteVertices]\config \coloneqq \Stab[\FiniteVertices](\config, 0_\FiniteVertices)$, and similarly we define $\Odom[\FiniteVertices]\config \coloneqq \Odom[\FiniteVertices](\config, 0_\FiniteVertices)$. We can stabilize from a state defined on a larger vertex set $W \supseteq \FiniteVertices$, that is from configuration $\config\colon W \to \N \cup \{\s\}$ and odometer $\odom \colon W \to \N$, by letting $\Stab[\FiniteVertices](\config, \odom) \coloneqq \Stab[\FiniteVertices](\config|_\FiniteVertices, \odom|_\FiniteVertices)$ and similarly for $\Odom[\FiniteVertices](\config, \odom)$. We let $x \in \Stab[\FiniteVertices]\config$ denote that $(\Stab[\FiniteVertices]\config)(x) = \s$.

For a configuration $\config\colon \Vertices \to \N \, \cup \, \{\s\}$ on all of $\Vertices$ and fixed instructions $\instr$, the stabilizing odometer is non-decreasing in $\FiniteVertices$. Therefore, we may define the possibly infinite odometer $\odom_{\config}\colon \Vertices \to \N \cup \{\infty\}$ for the system on the entire graph, 
$$
    \odom_{\config} = \Odom[\Vertices]\config \coloneqq \lim_{\FiniteVertices \nearrow \Vertices} \Odom[\FiniteVertices]\,\config.
$$
Above and throughout, when we take $\FiniteVertices \nearrow \Vertices$ it is over finite sets $\FiniteVertices \subset \Vertices$. 

It was demonstrated in \cite{AmirGurel-Gurevich10, rt} that particle fixation is equivalent to site and odometer fixation in the sense of \thref{lem:site-fixation} below. Recall the definitions of fixation and staying active from the introduction. Below and throughout, we always use a measure where the starting configuration $\config$ is independent of the instructions $\instr$. We say that a configuration $\config \in \N^{\Vertices}$ fixates if ARW with initial configuration $\config$ does.

\begin{lemma}[Equivalence of particle and site fixation]\thlabel{lem:site-fixation}
For any initial configuration $\config \in \N^{\Vertices}$ drawn from a translation-ergodic initial distribution,
$$
\P(\config\text{ fixates})=\P(m_{\config}(\origin)<\infty)\in \{0,1\}.
$$
\end{lemma}

Therefore, when $\config \in \N^{\Z^d}$ fixates, there is a limiting stable configuration $$\Stab\, \config \coloneqq \lim_{\FiniteVertices \nearrow \Vertices}\Stab[\FiniteVertices]\config,$$
since for each $x \in \FiniteVertices$ it can be shown that $x \in \Stab[\FiniteVertices]\config$ if and only if the stabilizing odometer ends on a sleep instruction.

\emph{Conservation of mass}, proved in \cite{AmirGurel-Gurevich10}, states that the final stable configuration $\Stab\,\config$ has the same density as the starting configuration $\config$.

\begin{prop}[Conservation of mass]\thlabel{prop:mass-conservation}
    For any i.i.d.\ active initial configuration $\config \in \N^{\Z^d}$, if $\config \sim \initdistr$ fixates a.s., then $\P(\origin \in \Stab \,\config) = \density.$
\end{prop}

We then get the following method to upper bound the critical density. When we take the $\limsup$ of some quantity $f(\FiniteVertices) \in \R$, we do so over finite sets $\FiniteVertices \subset \Z$:
$$
    \limsup_{\FiniteVertices \nearrow \Vertices}f(\FiniteVertices) \coloneqq \lim_{\text{finite } \FiniteVertices \nearrow \Vertices} \; \sup_{\text{finite } \SubFiniteVertices 
    \supseteq \FiniteVertices} f(U).
$$

\begin{prop}\thlabel{prop:unimodular-ub}
    $$
        \critdens  \le \limsup_{\FiniteVertices \nearrow \Vertices} \; \sup_{\text{active } \configg \in \N^{\Vertices}} \P(\origin \in \Stab[\FiniteVertices] \configg).
    $$
\end{prop}
\begin{proof}
    Take any subcritical density $\density < \critdens$ and let $\config \in \N^{\Vertices}$ be an i.i.d.\ active initial configuration with density $\density$. Then by conservation of mass, \begin{align*}
        \density = \P(\origin \in \Stab\, \config) &= \lim_{\FiniteVertices \nearrow \Vertices}\P(\origin \in \Stab[\FiniteVertices]\config) \\&\le \limsup_{\FiniteVertices \nearrow \Vertices}\sup_{\configg \in \N^{\Vertices}} \P(\origin \in \Stab[\FiniteVertices] \configg).
    \end{align*}
\end{proof}

We say a configuration $\config$ \emph{fills} a set of vertices $\SubFiniteVertices \subseteq \Vertices$ if it contains at least one active particle at each site of $\SubFiniteVertices$. The following lower bound on the critical density was proved in \cite[Lemma 10]{meanfield}:

\begin{prop}\thlabel{prop:fill}
    For any finite set $A \subset \Vertices$, 
    $$
        \critdens \ge \liminf_{\FiniteVertices \nearrow \Vertices}\inf_{\config \text{ fills } A} \P(\origin \in \Stab[\FiniteVertices] \config).
    $$
\end{prop}

\begin{remark}
    We conjecture that the upper bound from \thref{prop:unimodular-ub} is in fact an equality, as is the lower bound from \thref{prop:fill} after taking $A \nearrow \Z^d$. This is in line with the self-organized criticality picture believed to hold for ARW \cite{LS24}.
\end{remark}

\section{Weak and strong stabilization}\label{sec:weak-stab}

As evidenced by Propositions~\ref{prop:unimodular-ub} and~\ref{prop:fill}, the critical density is tied to the quantity $\P(\origin \in \Stab[\FiniteVertices] \config)$ for different configurations $\config$ and large finite sets $\FiniteVertices \subset \Vertices$ containing the origin. The procedure of strong stabilization via successive weak stabilization described below---developed in \cite{StaufferTaggi18} and refined in \cite{Taggi2019, rollaSurvey}---helps us analyze the probability the origin is occupied after stabilization. 

We define weak and strong stabilization on any finite $\FiniteVertices \subset \Vertices$ with respect to some subset $\SubFiniteVertices \subseteq \FiniteVertices$. In weak stabilization, a single active particle is frozen at any site of $\SubFiniteVertices$ if ever one enters. More precisely, following the definition of the usual stabilization, we say a site $x$ is \emph{weakly unstable} with respect to $\SubFiniteVertices$ if either $x \in U$ and multiple active particles are situated there, or $x \notin U$ and $x$ is unstable in the usual sense. The abelian property still holds by the same proof \cite{StaufferTaggi18} (alternatively, one can see this by coupling weak stabilization with the usual stabilization where $\lambda=\infty$ on $\SubFiniteVertices$). By the abelian property then, there is a well-defined \emph{weakly stable configuration} $\weakconfig$ and \emph{weakly stabilizing odometer} $\weakodom$ with respect to $\SubFiniteVertices$, which one
obtains by toppling weakly unstable vertices starting at state $(\config, \odom)$ until a weakly stable configuration is reached.

In the strong stabilization with respect to $\SubFiniteVertices$, no sleeping particles are allowed to remain on $\SubFiniteVertices$. Instead they are always \emph{acceptably} toppled, where acceptably toppling a sleeping particle first activates it and then executes the next instruction at its site (toppling an active particle is distinguished as a \emph{legal} toppling). Formally, $x \in U$ is \emph{strongly unstable} if $x$ contains any active or sleeping particles, and again off of $U$ the notions of unstable and strongly unstable agree. The abelian property again holds \cite{rollaSurvey} (now it is as if $\lambda=0$ on $U$), and we define the \emph{strongly stable configuration} $\strongconfig$ and the \emph{strongly stabilizing odometer} $\strongodom$ with respect to $\SubFiniteVertices$.

Weakly or strongly stabilizing with respect to a vertex $x \in \FiniteVertices$ means with respect to $\{x\}$. Whenever we omit the reference to $\SubFiniteVertices$, it is implied we are weakly or strongly stabilizing with respect to $\origin$.

\subsection{Strong stabilization via successive weak stabilization}\label{subsec:strong-stab-via-weak-stab}

To check if $\origin \in \Stab[\FiniteVertices] \config$, one can weakly stabilize with respect to $\origin$, and if this ends with an active particle at $\origin$, then with probability $\lambdasleep$ it immediately falls asleep, completing the stabilization with the origin occupied. If instead it jumps out of $\origin$, one can weakly stabilize again, and if a particle reenters there is another sleep trial, allowing for another chance to have $\origin \in \Stab[\FiniteVertices]\config$. This iterative scheme was developed in \cite{StaufferTaggi18}, and in \cite{Taggi2019} the number of sleep trials was decoupled from the outcomes by strongly stabilizing, jumping the particle out of the origin regardless of whether it fell asleep. 

More precisely, we define an iterative three-step procedure to strongly stabilize with respect to $\origin$ via successive weak stabilizations, starting from active configuration $\config \in \N^\FiniteVertices$ and odometer $0_\FiniteVertices$ (we may also allow configurations $\config \in \N^{\Z^d}$ on all of $\Z^d$ by first restricting to $\FiniteVertices$). Begin with a pre-step by weakly stabilizing once, leading to the weakly stable configuration $\weakconfig_1 = \weakconfig_1(\config, \FiniteVertices)$ and odometer $\weakodom_1 = \weakodom_1(\config, \FiniteVertices)$. By the abelian property, $\weakconfig_1$ and $\weakodom_1$ are determined solely by $\instr$, independent of the order sites are toppled. If the origin is now empty, the strong stabilization is complete. Otherwise there is a single active particle at the origin, and we enter iteration $j=1$:

\begin{samepage}
\begin{enumerate}
    \item (The \emph{jump-out}) Acceptably topple the particle at $\origin$ until it jumps out.
    \item Weakly stabilize with respect to $\origin$, leading to state $(\weakconfig_{j+1}, \weakodom_{j+1})$. 
    \item If the origin is empty, stop. Otherwise, enter iteration $j+1$, repeating steps 1--3.
\end{enumerate}
\end{samepage}

We define the number of \emph{chances} $\Ch = \Ch(\config, \FiniteVertices) \ge 0$ as the number of iterations completed, a random variable measurable with respect to $\instr$. Note that $\weakconfig_{\Ch+1}$ is the strongly stable configuration. For $1 \le j \le \Ch$, let the $j^\text{th}$ sleep trial, $\sleepoutcome_j \in \{0,1\}$, denote whether the first instruction in the $j^\text{th}$ jump-out step was a sleep instruction. Letting $\psi_\Ch(s)$ denote the probability-generating function $\E[\Ch^s]$, we show that $\P(\origin \notin \Stab[\FiniteVertices]\config) = \psi_\Ch(\lambdajump)$. For a configuration $\configg$ with $\configg(\origin)=1$, let $T^{\s, \origin}\configg$ denote the configuration obtained from $\configg$ by setting $(T^{\s, \origin}\configg)(\origin) = \s$.

\begin{lemma}\thlabel{lem:indie-trials}
Let $\config \in \N^\FiniteVertices$ be an active configuration on some finite set $\FiniteVertices \subset \Z^d$ containing the origin. Conditioned on the weakly stable configurations $\vec{\config}^{\,\Weak} \coloneqq (\weakconfig_1, \ldots, \weakconfig_\Ch)$, the sleep trials $\vec{b} = (\sleepoutcome_1, \ldots, \sleepoutcome_\Ch)$ are independent $\Ber(\lambdasleep)$ random variables. If there is some minimal $1 \le j \le \Ch$ with $b_j = 1$, then $$\Stab[\FiniteVertices] \config = T^{\s, \origin}\weakconfig_j,$$ and $\origin \in \Stab[\FiniteVertices]\config$. Otherwise, $$\Stab[\FiniteVertices] \config = \weakconfig_{\Ch+1},$$ and $\origin \notin \Stab[\FiniteVertices]\config$. Therefore, 
\begin{equation}\label{eq:pgf}
    \P(\origin \in \Stab[\FiniteVertices]\config) = \P\bigg(\sum_{j=1}^\Ch \sleepoutcome_j \ge 1\bigg) = 1 - \psi_{\Ch}(\lambdajump).
\end{equation}
\end{lemma}
    
\begin{proof}
    The idea is that the sleep instructions at the origin have no effect on $\vec{\config}^{\,\Weak}$.
    
    Let $\instr^\iota$ be obtained from $\instr$ by replacing every sleep instruction at the origin with the neutral instruction $\neutral$: that is $\instr^\iota_\origin(k) = \neutral$ whenever $\instr_\origin(k) = \s$, and all other instructions agree. Executing a neutral instruction does nothing to the configuration. It is easy to see that $\instr^\iota$ and $\instr$ produce the same weakly stable configurations: during any weak stabilization step, sleep instructions at the origin do not modify the configuration. Therefore, the configurations corresponding to each strong stabilization procedure may only differ in the middle of a jump-out step, temporarily disagreeing on whether the particle at the origin is asleep or awake. So $\vec{\config}^{\,\Weak}$ depends on $\instr$ only through $\instr^\iota$, and sleep instructions from $\instr$ at the origin indeed have no effect on $\vec{\config}^{\,\Weak}$. Thus, the $\sleepoutcome_j$'s conditioned on $\vec{\config}^{\,\Weak}$ are independent $\Ber(\lambdasleep)$ trials.

    Until there is a successful sleep trial, all topplings under instructions $\instr$ are legal, and therefore part of the stabilization of $\config$ on $\FiniteVertices$. If there is a first successful sleep trial $\sleepoutcome_j$, the stabilization is then complete, with $\Stab[\FiniteVertices] \config = T^{\s, \origin}\weakconfig_j$. Otherwise, the stabilization ends at $\Stab[\FiniteVertices] \config = \weakconfig_{\Ch + 1}$. Thus, indeed $\origin \in \Stab[\FiniteVertices] \config$ if and only if $\sum_{j=1}^\Ch \sleepoutcome_j \ge 1$. The second equation in \eqref{eq:pgf} is an easy general identity for probability-generating functions.
\end{proof}

The following bound on $\Ch$ was proved in \cite{Taggi2019}.

\begin{lemma}\thlabel{lem:expected-chances}
For any active configuration $\config$ and finite $\FiniteVertices \subset \Vertices$ containing the origin, 
\begin{equation*}
	\E[\Ch(\config, \FiniteVertices)] \le \Green[(\Graph, P)].
\end{equation*}
\end{lemma}

We obtain the following as a corollary.

\begin{corollary}\thlabel{cor:expected-visits} $\critdens(\Graph, P, \lambda) \le \Green[(\Graph, P)]\lambdasleep$.
\end{corollary}
\begin{proof}
    For any finite $\FiniteVertices$ containing $\origin$ and active configuration $\config \in \N^{{\Vertices}}$, 

    \begin{align*}
        \P(\origin \in \Stab[\FiniteVertices]\config) 
            &= \P\bigg(\sum_{j=1}^\Ch b_i \ge 1\bigg) && \text{(by equation \eqref{eq:pgf})} \\
            &\le \E\sum_{j=1}^\Ch b_i && \text{(by Markov's inequality)}\\
            &= \E[\Ch]\lambdasleep && \text{(by \thref{lem:indie-trials})}\\
            &\le  \Green[(\Graph, P)]\lambdasleep. && \text{(by \thref{lem:expected-chances})}
    \end{align*}

    By \thref{prop:unimodular-ub}, this upper bounds the critical density.
\end{proof}

\section{Proofs in the low-sleep-rate regime}\label{sec:proofs-low-lambda}

In this section we prove \thref{thm:low-sleep-rate}, beginning with \thref{lem:general-lb}.

\subsection{Proof of \thref{lem:general-lb}} Recall that $q = q(\Z^d, P, \lambda)$ denotes the probability a random walk from the origin falls asleep at the origin. To quantify $q$, modify the walk so it never falls asleep at the origin, and let $\pesc[\lambda]$ be the probability it falls asleep before visiting the origin for a second time. By the memorylessness of the walk, the number of origin visits of the modified walk has $\Geom(\pesc[\lambda])$ distribution. Focusing back on the unmodified walk, and using the independence of the sleep clock and the walk's trajectory, \begin{equation}\label{eq:q}
	q = 1 - \psi_{\Geom(\pesc[\lambda])}(\lambdajump) = \frac{\lambdasleep}{\lambdasleep + \lambdajump\pesc[\lambda]}.
\end{equation}
	
Now take finite $\FiniteVertices \subset \Z^d$ containing the origin, and let $\config$ be any configuration with $\config(\origin) = 1$, and therefore $\Ch = \Ch(\config, \FiniteVertices) \ge 1$. After the first jump-out, topple the particle until it either returns to the origin, falls asleep, or exits $\FiniteVertices$. Let $\ReturnEv{1}$ be the event that it returns to the origin before falling asleep or leaving $\FiniteVertices$, and thus $\Ch \ge 2$. Let $$\pesc[\lambda, \FiniteVertices] \coloneqq 1 - \P(\ReturnEv{1}),$$ which does not depend on $\config$. Similarly, on $\Ch \ge j$, let $\ReturnEv{j} \subseteq \{\Ch \ge j\}$ denote the event that after the $j^\text{th}$ jump-out, the particle returns to $\origin$ before falling asleep or leaving $\FiniteVertices$, ensuring that $\Ch \ge j + 1$. Define $\Ch'$ to be minimal so that $\ReturnEv{\Ch'}$ fails, so $\Ch' \le \Ch$, and
$$
    \Ch' \sim \Geom\big(\pesc[\lambda, \FiniteVertices]\big).
$$
Since $\lambdajump \in (0,1)$ we have $\psi_{\Ch'}(\lambdajump) \ge \psi_{\Ch}(\lambdajump)$. Therefore, by equation \eqref{eq:pgf}, 
\begin{align}
	\nonumber \P(\origin \in \Stab[\FiniteVertices] \config) = 1 - \psi_{\Ch}(\lambdajump)
	\nonumber &\ge 1 - \psi_{\Ch'}(\lambdajump)\\
	\label{eq:q-V}&= \frac{\lambdasleep}{\lambdasleep + \lambdajump\pesc[\lambda, \FiniteVertices]}.
\end{align}

Clearly, \begin{equation}\label{eq:preturn-limit}
	\lim_{\FiniteVertices \nearrow \Vertices} \pesc[\lambda, \FiniteVertices] = \pesc[\lambda].
\end{equation}

Therefore, 

\begin{align*}
	\critdens  &\ge \liminf_{\FiniteVertices \nearrow \Vertices}\inf_{\config(\origin)=1} \P(\origin \in \Stab[\FiniteVertices] \config) && \text{(by \thref{prop:fill})}\\
	&\ge \frac{\lambdasleep}{\lambdasleep + \lambdajump\pesc[\lambda]} && \text{(by equations \eqref{eq:q-V} and \eqref{eq:preturn-limit})}\\
	&= q, && \text{(by equation \eqref{eq:q})}
\end{align*}
completing the proof.

\subsection{Proof of \thref{thm:low-sleep-rate}}
    For any $(Z^d, P)$, let $S_n$ be a random walk started at the origin under $P$. Then as $\lambda \searrow 0$, using that $\lambdasleep = \lambda/(1 + \lambda) = \lambda(1 + o(1))$, 
	 \begin{align}
	 	\nonumber\critdens(\Graph, P, \lambda) &\ge q(\Graph, P, \lambda) && \text{(by \thref{lem:general-lb})}\\
	 	\nonumber&= \sum_{n=0}^\infty \lambdasleep \lambdajump^{n}\P(S_n = \origin)\\
	 	\label{eq:rw-low-sleep-rate}&=\lambda  \big(1+ o(1)\big)\sum_{n=0}^\infty \lambdajump^{n}\P(S_n = \origin).
	\end{align}

    We prove claims~\ref{thm:low-sleep-rate:Z2}--\ref{thm:low-sleep-rate:recurrent}.

\begin{enumerate}[label=(\roman*)]

    \item
	
	For SSRW on $\Z^2$, choose $C_2 > 0$ such that for $n \ge 1$, $\P(S_{2n} = \origin) \ge C_2/n$. Then,
	\begin{equation}\label{eq:low-lambda-Z2}
		\sum_{n=0}^\infty\lambdajump^{n}\P(S_n = \origin) 
            \ge  \sum_{n = 1}^\infty \lambdajump^{2n}\frac{C_2}{n}
            = C_2\log \bigg(\frac{1}{1 - \lambdajump^2}\bigg).
	\end{equation}
    Also, 
    $$1 - \lambdajump^2 = (1 + \lambdajump)\lambdasleep < 2\lambdasleep < 2 \lambda.$$
	Combining this with equations \eqref{eq:rw-low-sleep-rate} and \eqref{eq:low-lambda-Z2} we get
    \begin{align*}
        \critdens(\Z^2, \lambda) \ge \lambda(1 + o(1))C_2\log\bigg(\frac{1}{2 \lambda}\bigg) = \Theta(\lambda\log \lambda^{-1}).
    \end{align*}

    \item[(iii)]
    
	By the monotone convergence theorem, as $\lambda \searrow 0$ and thus $\lambdajump \nearrow 1$,
	\begin{equation}\label{eq:mct}
		\sum_{n=0}^\infty \lambdajump^{n}\P(S_n = \origin) \nearrow \sum_{n=0}^\infty \P(S_n = \origin) = \Green[(\Graph, P)].
	\end{equation}
	If $(\Graph, P)$ is recurrent then $\Green[(\Graph, P)] = \infty$, so equations \eqref{eq:rw-low-sleep-rate} and \eqref{eq:mct} yield $\critdens(G, P, \lambda) = \omega(\lambda)$.
	
    \vspace{\baselineskip}

    \item[(ii)]

    For transient $(\Graph, P)$, $\Green[(\Graph, P)] < \infty$. So \eqref{eq:rw-low-sleep-rate} and \eqref{eq:mct} give that
	\begin{align*}
		\critdens(G, P, \lambda) 
			&\ge \lambda\big(1 + o(1)\big)\big(\Green[(\Graph, P)] + o(1)\big) \\
			&= \Green[(\Graph, P)] \lambda + o(\lambda).
	\end{align*}
    The upper bound follows from \thref{cor:expected-visits}.

\end{enumerate}

\begin{remark}
    Arguing similarly as for $\critdens(\Z^2, \lambda)$, one can show $\critdens(\Z, \lambda) = \Omega(\sqrt{\lambda})$. This bound is already known to be tight, but \thref{lem:general-lb} admits a relatively elementary proof of the fact.
\end{remark}

\section{Proofs in the high-sleep-rate regime}\label{sec:proofs-high-lambda}

As $\lambda \to \infty$, we treat $\lambdajump$ like $\lambda^{-1}$ using the following bounds:
\begin{align}
    \lambdajump &< \lambda^{-1} \label{eq:lambdaj-ub},\\
    \lambdajump &= \lambda^{-1}(1 + o(1)). \label{eq:lambdaj-order}
\end{align}

\subsection{Proof of \thref{thm:high-sleep-rate}}

As with low sleep rates, we stochastically lower bound the number of sleep trials by a geometric random variable. For low $\lambda$, the extra sleep trial was guaranteed by the event that the particle jumped out of the origin returned without executing a sleep instruction. For high $\lambda$, the additional trial occurs if the the particle returns without visiting an empty site, so again it has no chance to fall asleep.

We fix a large ball $B_r$, and define an iterative \emph{carpet procedure}, where each step begins with an active particle on $\origin$ and a single active or sleeping particle at each site of $B_r \setminus \{\origin\}$. The procedure continues, ensuring an additional sleep trial, as long as all active particles on $B_r \setminus \{\origin\}$ immediately fall asleep, and the particle jumped out of the origin subsequently returns without exiting $B_r$. 

The proof of \thref{thm:high-sleep-rate-d-1} will use a similar but more delicate analysis.

\subsubsection{The carpet procedure}
	
	Fix $r \ge 0$. Take any finite $\FiniteVertices \subset \Vertices$ containing $B_r$ as well as some active configuration $\config \colon \Vertices \to \N$ that fills $B_r$ with active particles. We define a procedure which partially executes the strong stabilization of $\config$ with respect to $\origin$. We then define $\Ch' = \Ch'(\config, \FiniteVertices) \ge 1$ to be the number of partially completed iterations, which will serve as a lower bound for $\Ch$.
	
	Before entering the first iteration of the procedure, weakly stabilize with respect to $B_r$, leaving $B_r$ filled. Now execute the following steps until the procedure quits, starting at iteration $j=1$. Each iteration will begin with a configuration that is stable outside of $B_r$, has an active particle at $\origin$, and a single active or sleeping particle at each other site of $B_r$. Step 1 takes place during the $j^\text{th}$ weak stabilization, and step 2 during the $(j+1)^\text{st}$.
	
	\begin{enumerate}
		\item (Establish the carpet) Topple each active particle on $B_r \setminus \{\origin\}$ once. If any particle jumps, quit the procedure. Otherwise the $j^\text{th}$ weak stabilization is complete with no empty sites on $B_r$. 
		\item (Attempt a return) Jump the particle out of $\origin$, and then legally topple it until it returns to $\origin$ or escapes $B_r$. If it escapes, end the procedure. If it returns, which ensures $\Ch=\Ch(\config, \FiniteVertices) \ge j+1$, go back to step $1$ and complete iteration $j+1$.
	\end{enumerate}
	
	Letting $\Ch'$ be the iteration number when the procedure ends, we have $\Ch' \le \Ch$ by the observation in step 2, and since $\Ch \ge 1$ as $\config$ fills $\origin \in B_r$.
    
    In step 1, there are at most $|B_r| - 1$ particles that are toppled, so by a union bound its failure probability is at most $(|B_r| - 1)\lambdajump < (|B_r| - 1)\lambda^{-1}$. And step 2 always fails with the same probability, which is independent of $\config$, $\FiniteVertices$, and $\lambda$. We denote it by $\pesc[r]$. Therefore, with 
	\begin{equation*}
		\pesc[r, \lambda] \coloneqq  \min\!\big(1, (|B_r| - 1)\lambda^{-1} + \pesc[r]\big),
	\end{equation*}
	we have 
	 \begin{equation}\label{eq:geom-lb}
	   \Geom\big(\pesc[r, \lambda]\big) \preceq_{s.t.} \Ch' \preceq_{s.t.} \Geom\big(\pesc[r]\big). 
	\end{equation}
    
    \subsubsection{Relating the carpet procedure to the critical density}
    
    For any configuration $\config$ filling $B_r$ and finite $\FiniteVertices \supseteq B_r$, we have 
	\begin{align*}
		\P(\origin \notin \Stab[\FiniteVertices]\config) 
		&= \psi_\Ch(\lambdajump) && \text{(by equation \eqref{eq:pgf})}\notag \\
		&\le \psi_{\Ch'}(\lambdajump) && \text{(since $\Ch \ge \Ch'$)}\notag \\
		&\le \P(\Ch'=1)\lambdajump + \lambdajump^2 &&\text{(since $\Ch' \ge 1$)}\notag \\
		&< \pesc[r, \lambda]\lambda^{-1} + \lambda^{-2} &&\text{(by equations \eqref{eq:lambdaj-ub} and \eqref{eq:geom-lb})} \notag \\
        &\le \pesc[r]\lambda^{-1} + |B_r|\lambda^{-2}.\label{eq:first-order-ch-lb}
	\end{align*}
    
    Combining this with the following rearrangement of \thref{prop:fill}, 
    \begin{equation}\label{eq:prop-fill-inverted}
        1 - \critdens \le \limsup_{\FiniteVertices \nearrow \Vertices}\sup_{\config \text{ fills } B_r} \P(\origin \notin \Stab[\FiniteVertices] \config),
    \end{equation}
    we get 
	\begin{equation}\label{eq:high-lambda-first-order}
		1 - \critdens \le \pesc[r]\lambda^{-1} + |B_r|\lambda^{-2} = \big(\pesc[r] + o(1)\big)\lambda^{-1}.
	\end{equation}
    
    Letting $r \nearrow \infty$ and thus $\pesc[r] \searrow \pesc$, we have
    
    \begin{equation}\label{eq:pesc-first-order}
		1 - \critdens \le \big(\pesc + o(1)\big) \lambda^{-1}.
	\end{equation}

    \subsubsection{Recurrent walks}

    \thref{thm:high-sleep-rate}~\ref{thm:high-sleep-rate:recurrent} follows immediately from \eqref{eq:pesc-first-order}, as $\pesc=0$ for recurrent walks.

    \subsubsection{Dimension 2}
	
	Take $r= \sqrt[]{\lambda/\log \lambda}$. For the particle in the carpet procedure to escape $B_r$, it must take at least $r$ steps before returning to the origin, and so, letting $\tau_\origin^+$ denote the first return time for SSRW on $\Z^2$, \begin{equation}\label{eq:tau}
		\pesc[r] \le \P(\tau_\origin^+ > r).
	\end{equation}
	A classical result states that in dimension 2, $\tau_\origin^+$ has logarithmically decaying tail \cite{erdos}. Therefore, by \eqref{eq:high-lambda-first-order} and \eqref{eq:tau}, 
	\begin{align*}
		1 - \critdens(\Z^2, \lambda) &\le  \P(\tau_\origin^+ \ge r)\,\lambda^{-1} + |B_r|\lambda^{-2}\\
		&= O( \log^{-1}\! r)\lambda^{-1} + O((\lambda\log \lambda)^{-1})\\ 
		&= O((\lambda\log \lambda)^{-1}),
	\end{align*}
	establishing \thref{thm:high-sleep-rate}~\ref{thm:high-sleep-rate:Z2}.

    \subsubsection{Transient walks}
    From Equation \eqref{eq:pesc-first-order} we have
    \begin{equation*}
        \critdens \ge 1 - \pesc \lambda^{-1} + o(\lambda^{-1}),
    \end{equation*}
    the lower bound of \thref{thm:high-sleep-rate}~\ref{thm:high-sleep-rate:recurrent}. The outline of the argument for the upper bound is as follows. For $\config$ that fills $B_r$, we show that $\Ch' = \Ch'(\config, \FiniteVertices)$ approximates $\Ch = \Ch(\config, \FiniteVertices)$ well when $r$ is large. Then $\P(\Ch=1)$ is close to $\pesc$, so $\P(\origin \notin \Stab[\FiniteVertices] \config)$ is not much less than $\pesc\lambda^{-1}$. Now an i.i.d.\ configuration $\configg$ with density larger than $1 - \pesc\lambda^{-1}$ likely fills $B_r$, and so loses mass after stabilization. Therefore, $1 - \pesc \lambda^{-1}$ must be supercritical, as otherwise conservation of mass would be violated.

    \begin{proof}[Proof of \thref{thm:high-sleep-rate}~\ref{thm:high-sleep-rate:transient}]
    
    For any $\config$ filling $B_r$ and finite $\FiniteVertices \supseteq B_r$, we have 
    \begin{align}
		\P(\Ch\neq \Ch') 
		&= \P(\Ch - \Ch' \ge 1) \notag \\
		&\le \E[\Ch - \Ch'] && \text{(by Markov's inequality)} \notag \\
		&\le \frac{1}{\pesc} - \frac{1}{\pesc[r, \lambda]}. && \text{(by  \thref{lem:expected-chances} and equation \eqref{eq:geom-lb})}  \label{eq:ch'-approx-ch}
	\end{align}
    Thus, 
    \begin{align}
		\P(\origin \notin \Stab[\FiniteVertices]\config) 
			&\ge \lambdajump \P(\Ch = 1) && \text{(by equation \eqref{eq:pgf})}\notag \\
			&\ge \lambdajump \big(\P(\Ch' = 1) - \P(\Ch \neq \Ch')\big) \notag \\
			&\ge \lambdajump\bigg(\pesc[r] - \frac{1}{\pesc} + \frac{1}{\pesc[r, \lambda]}\bigg). &&\text{(by equations \eqref{eq:geom-lb} and \eqref{eq:ch'-approx-ch})}\label{eq:transient-ub}
	\end{align}	

    Let 
    $$
        \alpha_r \coloneqq \pesc[r] - \frac{1}{\pesc} + \frac{1}{\pesc[r]}.
    $$
    Fix $0 < \alpha < \alpha_r$, and let $\configg$ be the i.i.d.\ Bernoulli configuration with density $1 - \alpha \lambda^{-1}$, where $\lambda$ is large enough that $1 - \alpha \lambda^{-1} \ge 0$. If we had $\critdens(\Z^d, P, \lambda) > 1 - \alpha \lambda^{-1}$, by conservation of mass (\thref{prop:mass-conservation}) as $\FiniteVertices \nearrow \Vertices$ we would have 
    \begin{equation}\label{eq:c-o-m}
        \P(\origin \notin \Stab[\FiniteVertices] \configg) \to \P(\origin \notin \Stab\,\configg) = \alpha \lambda^{-1}.
    \end{equation}
    But $\configg$ fills $B_r$ with at least probability $1 - \alpha \lambda^{-1}|B_r|$. Thus, by Equation \eqref{eq:transient-ub}, 
    \begin{align*}
        \P(\origin \notin \Stab[\FiniteVertices]\configg) 
			&\ge (1 - \alpha \lambda^{-1}|B_r|) \cdot \lambdajump \cdot \bigg(\pesc[r] - \frac{1}{\pesc} + \frac{1}{\pesc[r, \lambda]}\bigg)\\
            &= (1 + o(1)) \cdot \lambda^{-1}(1 + o(1)) \cdot (\alpha_r + o(1))\\
            &= (\alpha_r + o(1))\lambda^{-1}.
    \end{align*}
    When $\lambda$ is sufficiently large this contradicts \eqref{eq:c-o-m}, so therefore we must have $\critdens \le 1 - \alpha \lambda^{-1}$. Taking $\alpha \nearrow \alpha_r$, we have $\critdens \le 1 - \alpha_r \lambda^{-1} + o(\lambda^{-1})$. Taking $r \nearrow \infty$, and thus $\alpha_r \to \pesc$, establishes the result. 
        
    \end{proof}

    \subsection{Proof of \thref{thm:high-sleep-rate-d-1}}
    
	To sketch the argument for $\critdens(\Z, \lambda)$, we jointly consider the probability of a large excursion with that of a sufficiently small carpet, deviating from the above procedure which had a fixed carpet size. This allows us to improve our bound on $1 - \critdens(\Z, \lambda)$ from $O(1/\lambda)$ to $O(\log \lambda / \lambda^{2})$. It also increases the carpet density: after the first weak stabilization, the carpet holes also must have $O(\log \lambda / \lambda^{2})$ density, as most likely the particle at the origin will immediately fall asleep, and the holes will become empty sites in the final stable configuration. After a subsequent weak stabilization, we are looking at the stable configuration conditioned on the probability $\sim\!\lambda^{-1}$ event the particle did not fall asleep. Thus the density of holes increases by no more than a factor of $\lambda$, and so is $O(\log \lambda / \lambda)$. From this we are able to get the improved bound $O(\log^2 \! \lambda / \lambda^{3})$. By successively bootstrapping each improvement, the bound diminishes to $O(\log^{n-1} \!\lambda / \lambda^{n})$ for any $n$.

    \begin{proof}[Proof of \thref{thm:high-sleep-rate-d-1}]
    
	As a motivating warmup, we perform the first improvement, $1 - \critdens  = O(\lambda^{-2} \log \lambda)$, separately. Take configuration $\config$ that fills $B_{\lambda}$ and finite $\FiniteVertices \supseteq B_{\lambda}$. By Equation ~\eqref{eq:pgf}, 
    \begin{equation}\label{eq:first-order-pgf}
        \P(\origin \notin \Stab[\FiniteVertices]\config) \le \P(\Ch = 1)\lambda^{-1} + \lambda^{-2}.
    \end{equation}
    We seek an upper bound 
    \begin{equation}\label{eq:c-11}\tag{$\ast$}
        \P(\Ch = 1) \le C \lambda^{-1}\log \lambda
    \end{equation}
    for some constant $C > 0$ not depending on $\FiniteVertices$ or $\config$. Then by \eqref{eq:first-order-pgf} we get 
    \begin{align}\label{eq:base-case}
        \P(\origin \notin \Stab[\FiniteVertices]\config)
            &\le C \,\lambda^{-2} \log \lambda + \lambda^{-2} \nonumber \\
            &\le C_1 \, \lambda^{-2} \log \lambda
    \end{align}
    for some properly chosen constant $C_1 > 0$ and all large $\lambda$. By \eqref{eq:prop-fill-inverted}, we also get $1 - \critdens \le C_1 \, \lambda^{-2} \log \lambda$.

    To obtain \eqref{eq:c-11}, weakly stabilize with respect to $\origin$, leading to the configuration $\weakconfig_1$. For each $x \in B_\lambda$, let $\Emp{x,1}$ denote the event that $\weakconfig_1(x) = 0$. For each $x \in B_{\lambda} \setminus \{\origin\}$, we have \begin{equation}\label{eq:emp-1}
		\P(\Emp{x,1}) \le \lambdajump < \lambda^{-1},
	\end{equation}
	as can be seen by first weakly stabilizing with respect to $x$ and the origin, and then attempting to end the full weak stabilization by putting the particle at $x$ to sleep. 

    Define the carpet $B_{\Carpet_1}$ as the maximal centered ball with integer radius $\Carpet_1 \le \lambda$ such that $\Emp{x,1}$ fails for each $x \in B_{\Carpet_1}$. After the weak stabilization, jump the particle out of the origin and legally topple it until it either returns to $\origin$---ensuring that $\Ch \ge 2$---or leaves the carpet. By gambler's ruin then, 
    \begin{equation}\label{eq:gamblers-ruin-1}
        \P(\Ch = 1 \mid \Carpet_1 = i) \le 1/(i+1).
    \end{equation}
    Also for any integer $0 \le i \le \lambda - 1$, if $\Carpet_1=i$ then $\Emp{x,1}$ holds for one of $x = \pm (i+1)$. By \eqref{eq:emp-1} then,
    \begin{equation}\label{eq:emp-x-1}
        \P(\Carpet_1 = i) < 2\lambda^{-1}.
    \end{equation}
     Combining \eqref{eq:gamblers-ruin-1} and \eqref{eq:emp-x-1},
	\begin{align*}
		\P(\Ch = 1) 
            &= \sum_{i = 0}^\lambda \P(\Ch = 1, \Carpet_1 = i)\\
			&\le \sum_{i=0}^{\lambda-1} 2\lambda^{-1}\frac{1}{i+1} + \P(\Carpet_1=\lfloor \lambda \rfloor )\frac{1}{\lfloor \lambda \rfloor  + 1}\\ 
			&\le 2\lambda^{-1}H_{\lfloor \lambda \rfloor} + \lambda^{-1}\\
			&\le C \lambda^{-1}\log \lambda
	\end{align*}
    for properly chosen $C > 0$, where $H_n$ is the $n^\text{th}$ harmonic number. This establishes \eqref{eq:c-11}.

    Now for $k \ge 1$, let
    $$
        r_k \coloneqq \sum_{j=0}^{k}\lambda^j.
    $$
    We inductively show that for each $k \ge 1$, there is a constant $C_k > 0$ such that for all large $\lambda$, 
    \begin{align}\label{eq:induction-case}
        \P(\origin \notin \Stab[\FiniteVertices]\config)
            \le C_k \lambda^{-(k+1)} \log^{k}\!\lambda 
    \end{align}
    holds for any $\config$ that fills $B_{r_k}$ and finite $\FiniteVertices \supseteq B_{r_k}$. Then, by \eqref{eq:prop-fill-inverted}, for large $\lambda$, 
    $$1 - \critdens(\Z, \lambda) \le C_k \, \lambda^{-(k+1)} \log^{k}\lambda = o(\lambda^{-k)}),$$  proving \thref{thm:high-sleep-rate-d-1}.

    We have established the $k=1$ case, although the $k=0$ case is immediate using $C_0 = 1$. Now take $k \ge 2$, and suppose we have demonstrated \eqref{eq:induction-case} at index $k-1$ for large $\lambda$. Using 
	\begin{align}
		\P(\origin \notin \Stab[\FiniteVertices] \config ) &\le \sum_{j=1}^{k}\P(\Ch=j)\lambdajump^j + \lambdajump^{k+1} \notag \\
		&\le \sum_{j=1}^{k}P(\Ch=j)\lambda^{-j} + \lambda^{-(k+1)}, \label{eq:k-order-pgf}
	\end{align}
    we seek to find for each $1 \le j \le k$, constants $C_{k,j} > 0$ such that, for all large $\lambda$,  
    \begin{equation}\tag{$\ast\ast$}\label{eq:g-kj}
        \P(\Ch=j) \le C_{k,j}\lambda^{j-(k+1)}\log^{k} \lambda
    \end{equation}
    holds for any finite $\FiniteVertices \supseteq B_{r_{k}}$ and $\config$ filling $B_{r_k}$. This then implies \eqref{eq:induction-case}, as \eqref{eq:k-order-pgf} is a finite sum. It remains to show \eqref{eq:g-kj}.

    On $\Ch \ge j$, let $\Emp{x,j}$ for $x \in B_{\lambda^k}$ denote the event that $\weakconfig_j(x) = 0$, and let $0 \le \Carpet_j \le \lambda^k$ be the maximal integer so that $\Emp{x,j}$ fails for all $x \in B_{\Carpet_j}$. Again, we wish to bound $\P(\Emp{x,j})$ and in turn $\P(\Ch = j)$, showing that on $\Ch \ge j$ it is likely the jumped-out particle returns without visiting an empty site, ensuring $\Ch \ge j + 1$. 
    
    Take $x \in B_{\lambda^k} \setminus \{\origin\}$. We have $B_{r_{k-1}}(x) \subseteq B_{r_k}$, since $r_{k} = r_{k-1} + \lambda^k$. Thus by the induction hypothesis and translation invariance, for sufficiently large $\lambda$, 
	\begin{equation}\label{eq:emp-x}
			\P(x \notin \Stab[\FiniteVertices]\config) \le C_{k-1}\lambda^{-k}\log^{k-1}\lambda.
	\end{equation}
	We use this to show that $\Emp{x,j}$ has low probability, as it creates an opportunity to have $x \notin \Stab[\FiniteVertices]\config$. Namely, by \thref{lem:indie-trials}, and taking $\lambda \ge 1$ so that $\lambdasleep \ge 1/2$ and $\lambdajump \ge (2\lambda)^{-1}$, 
	\begin{align*}
		\P(x \notin \Stab[\FiniteVertices]\config) \mid \Emp{x,j}) 
            &\ge \P(\sleepoutcome_1=0,\ldots, \sleepoutcome_{j-1}=0,\sleepoutcome_j=1\mid \Emp{x,j})\\
            &=  \lambdasleep\lambdajump^{j-1}\\
            &\ge 2^{-j}\lambda^{-j+1}.
	\end{align*}
	Combining this with \eqref{eq:emp-x} yields
	\begin{align*}
		C_{k-1}\lambda^{-k}\log^{k-1}\!\lambda \ge \P(x \notin \Stab[\FiniteVertices]\config, \Emp{x,j}) \ge \P(\Emp{x,j}) 2^{-j}\lambda^{-j+1},
	\end{align*}
	and, rearranging, 
	\begin{equation}\label{eq:emp-x-j}
		\P(\Emp{x,j}) \le 2^j C_{k-1}\lambda^{j - (k+1)}\log^{k-1}\!\lambda \eqqcolon f_{k,j}(\lambda).
	\end{equation}
    
    Repeating the argument from the $k=1$ warmup then, toppling the particle after the $j^\text{th}$ jump-out until it returns to $\origin$ or exits $B_{\Carpet_j}$, we get
    \begin{equation}\label{eq:conditional-on-edge-eq-i}
        \P(\Ch = j \mid \Ch \ge j, \Carpet_j = i) \le \frac{1}{i+1},
    \end{equation}
    and for $i \le \lambda^k-1$, using \eqref{eq:emp-x-j}, 
    \begin{equation}\label{eq:p-edge-eq-i}
        \P(\Ch \ge j, \Carpet_j = i) \le 2f_{k,j}(\lambda).
    \end{equation}
    Combining these, 
	\begin{align*}
		\P(\Ch = j) 
            &\le 2f_{k,j}(\lambda)\sum_{i=0}^{\lambda^k-1}\frac{1}{i+1} + \lambda^{-k}\\
            &= 2f_{k,j}(\lambda)H_{\lfloor \lambda^k \rfloor} + \lambda^{-k}\\
            &\le C_{k,j}\lambda^{j-(k+1)}\log^{k} \lambda
	\end{align*}
    for properly chosen $C_{k,j} > 0$ and all large $\lambda$, proving \eqref{eq:g-kj}. 
        
    \end{proof}

\section*{Acknowledgements}
We would like to thank Matt Junge for many helpful discussions and for his detailed feedback. We thank Leo Rolla for his help understanding the picture on general graphs, Nicolas Forien for posing the question regarding \thref{thm:high-sleep-rate-d-1} in Budapest, and Jacob Richey for organizing the workshop in Budapest and for interesting discussions. Kaufman was partially supported by NSF DMS Grants 2238272 and 2349366. Part of this research was conducted during the 2025 Baruch College Discrete Mathematics NSF Site REU.

\bibliographystyle{alpha}
\bibliography{BA.bib}

\end{document}